\documentclass{amsart}
\usepackage{euscript}           
\usepackage{amsmath,amsthm}     
\usepackage{amssymb}            
\usepackage[usenames,dvipsnames]{color}
\usepackage{graphicx}

\newcommand{\xa}{{\ref{3}}}
\newcommand{\xb}{{\ref{4}}}
\newcommand{\xc}{{\ref{5}}}
\newcommand{\xd}{{\ref{6}}}
\newcommand{\ya}{{\ref{a}}}
\newcommand{\yb}{{\ref{b}}}
\newcommand{\yc}{{\ref{c}}}
\newcommand{\yd}{{\ref{d}}}
\newcommand{\za}{{\ref{za}}}
\newcommand{\zb}{{\ref{zb}}}
\newcommand{\zc}{{\ref{zc}}}
\newcommand{\zd}{{\ref{zd}}}
\newcommand{\ze}{{\ref{ze}}}
\newcommand{\zf}{{\ref{zf}}}
\newcommand{\bfn}{{\mathbf{n}}}
\newcommand{\bfe}{{\mathbf{e}}}
\newcommand{\ms}{{multisemisimplicial}}
\numberwithin{equation}{section}
\newtheorem{thm}{Theorem}[section]
\newcounter{numerierer}
\newcounter{leer}

\newtheorem{defn}[thm]{Definition}
\newtheorem{prop}[thm]{Proposition}

\newtheorem{lemma}[thm]{Lemma}

\theoremstyle{definition}  
\newenvironment{definition}{\begin{defn}\rm}{\end{defn}}

\newtheorem{set theory}[thm]{Set Theoretic Prelude}

\newtheorem{remark}[thm]{Remark}
\usepackage[frame,matrix,curve, arrow]{xy}

\xymatrixcolsep{1.9pc}                          
\xymatrixrowsep{1.9pc}
\newdir{ >}{{}*!/-5pt/\dir{>}}                  
\newdir{ |}{{}*!/-5pt/\dir{|}}                  

\begin{document}

\title{On semisimplicial sets satisfying the Kan condition}
\author{James E.\ McClure}

\address{ Department of Mathematics, Purdue University, 150 N.\ University
Street, West Lafayette, IN 47907-2067}

\thanks{The author was partially supported by NSF grants. He
thanks the Lord for making his work possible.}


\date{October 19, 2012}

\begin{abstract}
A semisimplicial set has face maps but not degeneracies.  A basic fact, due 
to Rourke and Sanderson, is that a semisimplicial set satisfying the Kan 
condition can be given a simplicial structure.  The present paper gives a 
combinatorial proof of this fact and a generalization to multisemisimplicial 
sets.  
\end{abstract}
\maketitle

\section{Introduction}

A semisimplicial set
$X$ is a sequence of sets $X_n$ for $n\geq 0$ with maps 
$d_i:X_n\to X_{n-1}$ for $0\leq i\leq n$ satisfying
\begin{equation}
\label{e1}
d_id_j=d_{j-1}d_i \quad\text{if\, $i<j$}.
\end{equation}
Elements of $X_n$ are called {\it $n$-simplices}, and the maps $d_i$ are 
called {\it face maps}.

Semisimplicial sets occur in various areas of mathematics, especially
geometric topology, surgery theory (e.g., \cite{MR1211640}), and 
homological algebra (e.g., \cite{WEIB,Schwede}).

\begin{definition}
A semisimplicial set $X$ satisfies the {\it Kan condition} if, for every
collection of $n+1$ $n$-simplices $x_0,\ldots,x_{k-1},x_{k+1},\ldots,x_{n+1}$ 
satisfying
\[
d_ix_j=d_{j-1}x_i\quad\text{whenever $i<j$ with $i\neq k\neq j$},
\]
there is an $n+1$ simplex $x$ with $d_ix=x_i$ for all $i\neq k$.
\end{definition}

In \cite{RS}, Rourke and Sanderson used PL topology to prove that a 
semisimplicial set which satisfies the Kan condition has a simplicial 
structure:

\begin{thm}\cite[Theorem 5.7]{RS}
\label{1}
Let $X$ be a semisimplicial set satisfying the Kan condition.  Then there are
functions
\[
s_j:X_n\to X_{n+1}, 
\]
for $n\geq 0$ and $0\leq j\leq n$,
with the following properties.
\begin{alignat}{2}
\label{3}
d_is_j&=s_{j-1}d_i &\quad&\textrm{if\,  $i<j$}.\\
\label{4}
d_is_jx&=x &\quad&\textrm{if\,  $i=j,j+1$}.\\
\label{5}
d_is_j&=s_jd_{i-1} &\quad&\textrm{if\,  $i>j+1$}.\\
\label{6}
s_js_i&=s_is_{j-1} &\quad&\textrm{if\,  $i<j$}.
\end{alignat}
\end{thm}

Note that (\xd) is written in a slightly nonstandard form which is
equivalent to the usual one.  Also note that the theorem does not say that the
simplicial structure on $X$ is unique, nor does it give a functor
from Kan semisimplicial sets to simplicial sets.

The purpose of this note is to generalize Theorem \ref{1} to
multisemisimplicial sets, for use in \cite{LM13}.  It is not at all clear how 
to generalize the
geometric argument given by Rourke and Sanderson, so instead I will give a
combinatorial proof of Theorem \ref{1} which generalizes easily to the
multisemisimplicial setting.

The organization of the paper is as follows.  
The proof of Theorem \ref{1} is a
double induction which is carried out in Sections \ref{s2} and \ref{s3}.  An
auxiliary lemma is proved in Section \ref{s3.5}. 
The
generalization of Theorem \ref{1} to multisemisimplicial sets is stated in 
Section \ref{s4} and proved in Sections \ref{s5}--\ref{s7}.
Section \ref{s8} gives an application of Theorem \ref{1}, using it to give a
new proof of \cite[Corollary 5.4]{RS}.

\smallskip
I would like to thank Stefan Schwede for a helpful conversation.

\medskip

{\bf A note on terminology.}  The first appearance of 
semisimplicial sets in the literature was in \cite{EZ}, where they were called
semisimplicial complexes. The motivation for the name is that an ordered 
simplicial complex has two properties: (i) there are face maps satisfying
\eqref{e1} and (ii) a simplex is determined by its faces; a ``semisimplicial
complex'' has only the first property.  The same paper introduced the concept 
of ``complete semisimplicial complexes,'' in which the word ``complete'' 
referred to the presence of degeneracy maps.  During the 1950's and 1960's it 
became common to drop the word ``complete'' from the terminology (probably 
because all of the known applications required degeneracies) and to use the 
term semisimplicial complex to mean what is now called a simplicial set.

Semisimplicial complexes in the original sense were resurrected and renamed 
(as $\Delta$-sets) in \cite{RS}.  The name $\Delta$-set seems infelicitous 
because the category that governs simplicial sets is called $\Delta$.   In 
using the term semisimplicial set I am following the terminology of 
\cite[Definition 8.1.9]{WEIB}.





\section{Beginning of the proof of Theorem \ref{1}}
\label{s2}

We will construct the degeneracies $s_jx$ by a double induction on $\deg(x)$
and $j$.  Specifically, given $n\geq
0$ and $0\leq k\leq n$, we assume that $s_jx$ has been constructed
when $\deg(x)<n$ and also when $\deg(x)=n$ and $j<k$, and that properties
(\xa)--(\xd) hold in all relevant cases (that is, in all cases involving only
degeneracies that have already been constructed).  Let $x\in X_n$; we need to
construct $s_kx$.

There are two cases.  The easier case (which does not
use the Kan condition) is when $x$ is in the image of $s_j$ for some $j<k$.
We give the proof for this case in this section.

Choose the {\it smallest} $j$ for which $x$ is in the image of $s_j$; then
\[
x=s_j w
\]
for some $w$.  In fact this $w$ is unique, by the following lemma.

\begin{lemma}
\label{l1}
Degeneracy maps are monomorphisms.
\end{lemma}

\begin{proof}
This is immediate from (\xb).
\end{proof}

We now define $s_kx$ to be $s_js_{k-1}w$ (as required by (\xd)).

It remains to verify (\xa)--(\xd).  The verifications of (\xa)--(\xc) are
routine applications of the simplicial identities and are left to the reader.

For (\xd) we need a well-known fact:

\begin{lemma}
\label{l2}
If $s_jw=s_iy$ for some $i$ with $j<i<k$ then there is a $v$ with 
$y=s_jv$ and $w=s_{i-1}v$.
\end{lemma}

\begin{proof}
Let $v=d_j y$.  Then 
\[
s_iy=s_jw=s_jd_js_jw=s_jd_js_iy=s_js_{i-1}d_jy=s_js_{i-1}v=s_is_jv,
\]
so $y=s_jv$ by Lemma \ref{l1}.  Now 
\[
s_jw=s_iy=s_is_jv=s_js_{i-1}v
\]
so $w=s_{i-1}v$ by Lemma \ref{l1}.
\end{proof}

Now we verify (\xd).  Let $y\in X_{n-1}$ and let $i<k$. Choose the smallest 
$j$ for which $s_iy$ is in the image of $s_j$. If $j=i$ we are done, 
otherwise let $s_iy=s_jw$.  Let $v$ be the element given by Lemma \ref{l2}. 
Then
\begin{align*}
s_ks_iy&=s_js_{k-1}w\quad\text{by definition of $s_k$}
\\
&=s_js_{k-1}s_{i-1}v
=s_js_{i-1}s_{k-2}v=s_is_js_{k-2}v
\\
&=s_is_{k-1}s_jv\quad\text{because $j<i<k$ so $j<k-1$}
\\
&=s_is_{k-1}y,
\end{align*}
as required.

\section{Conclusion of the proof of Theorem \ref{1}}
\label{s3}

Next we must construct $s_k x$ in the remaining case, when $x$ is not
in the image of any $s_j$ with $j<k$.

First note that the simplicial identities determine {\it all} faces of $s_kx$,
so it is not possible to build $s_kx$ directly from the Kan condition.
Instead, we apply the Kan condition twice to construct a suitable element 
$T_kx$ in 
degree $\deg(x)+2$ and then define 
\begin{equation}
\label{sT}
s_kx=d_0T_kx.
\end{equation}

We will construct the elements $T_jx$ by a double induction on $\deg(x)$ and
$j$.
In order to see what properties 
we want $T_j$ to have in the inductive hypothesis,
we use a heuristic argument.  For $i<j$, we want the simplicial 
identity (\xa)
to hold.  The left-hand side of (\xa) will be equal to 
$d_0d_{i+1}T_j$ (using \eqref{sT} and $d_id_0=d_0d_{i+1}$) and the
right-hand side will be $d_0T_{j-1}d_i$.  The simplest way for the two sides 
to be equal is to have $d_{i+1}T_j=T_{j-1}d_i$, which we rewrite as
\begin{equation}
\label{a}
d_iT_j=T_{j-1}d_{i-1}  \quad\text{if\, $0<i<j+1$};
\end{equation}
this is the first of the properties we want. 
Similarly, the simplicial identity (\xc) leads to the equation
\begin{equation}
\label{b}
d_iT_j=T_jd_{i-2} \quad\text{if\, $i>j+2$}.
\end{equation}
Finally, the simplicial identity (\xb) leads to two equations:
\begin{equation}
\label{c} 
d_{j+1}T_j=d_{j+2}T_j \quad\text{for all $j$},
\end{equation}
and
\begin{equation}
\label{d} 
d_0d_{j+1}T_jx=x \quad\text{for all $j$ and $x$}.
\end{equation}

\begin{lemma}
\label{l3}
Let $X$ be a semisimplicial set satisfying the Kan condition.  Then there are
functions
\[
T_j:X_n\to X_{n+2}, 
\]
for $n\geq 0$ and $0\leq j\leq n$,
satisfying (\ya)--(\yd).
\end{lemma}

The proof will be given in the next section.  We can now finish the
proof of Theorem \ref{1}.  Given that $x$ is not in the image of $s_j$ for
any $j<k$, define $s_kx$ by \eqref{sT}.  The simplicial identity
(\xa) follows at once from (\ya), and (\xc) follows from (\yb).  To see that
$d_ks_kx=x$ we use (\yd), and then $d_{k+1}s_kx=x$ follows from (\yc).
The identity (\xd) is vacuous because of the assumption on $x$.
\qed

\section{Proof of Lemma \ref{l3}}
\label{s3.5}

To see how the construction of $T_j$ works it's helpful to begin with the case 
$\deg(x)=0$ and $j=0$.  Since we want (\yc) and (\yd) to hold we must have
$d_1T_0x=d_2T_0x$ and $d_0d_1T_0x=x$.  The Kan condition gives an
element $y$ of degree 1 with $d_0 y=x$.  A second application of the Kan
condition gives $T_0x$ with $d_1T_0x=d_2T_0x=y$.  Then (\yc) and (\yd) are 
immediate from the construction, and (\ya) and (\yb) are vacuous in this case.

Now assume that $T_jx$ has been constructed with properties (\ya)--(\yd) for
$\deg(x)<n$ and also for $\deg(x)=n$ and $j<k$.  Let $x\in X_n$; we need to
construct $T_k x\in X_{n+2}$.

We begin by constructing an element $y$ of degree $n+1$ which will play the
role of $d_{k+1}T_kx$.  We want $d_0y$ to be $x$ because of (\yd). For
$0<j<k+1$ we want $d_jy$ to be $d_kT_{k-1}d_{j-1}x$ because of (\ya), and for 
$j>k+1$ we want $d_jy$ to be $d_{k+1}T_kd_{j-1}x$ because of (\yb).  In order
to apply the Kan condition we need to show that these choices are consistent:

\begin{lemma}
\label{l4}
Given $x\in X_n$, let 
\begin{equation}
\label{e2}
y_j=
\begin{cases}
x & \text{if\, $j=0$,}\\
d_kT_{k-1}d_{j-1}x & \text{if\, $0<j<k+1$,}\\
d_{k+1}T_kd_{j-1}x & \text{if\, $j>k+1$.}
\end{cases}
\end{equation}
Then
\begin{equation}
\label{e3}
d_iy_j=d_{j-1}y_i \quad\text{whenever $i<j$ with $i\neq k+1\neq j$.}
\end{equation}
\end{lemma}

\begin{proof}[Proof of Lemma \ref{l4}]
First suppose $i=0$.  Then the right-hand side of \eqref{e3} is equal to 
$d_{j-1}x$, and the left-hand side simplifies to $d_{j-1}x$ by 
(\yd). 

Next suppose $0<i<j<k+1$. Both sides simplify to
$d_{k-1}T_{k-2}d_{j-2}d_{i-1}x$ by (\ya).

If $0<i<k+1<j$ then both sides simplify to $d_kT_{k-1}d_{j-2}d_{i-1}x$, 
using (\ya) on the left side and (\yb) on the right side.

Finally, if $k+1<i<j$ then both sides simplify to 
$d_{k+1}T_kd_{j-2}d_{i-1}x$ by (\yb).
\end{proof}

Now we apply the Kan condition to construct $y\in X_{n+1}$ with $d_jy=y_j$
for $j\neq k+1$.

Next we will use the Kan condition to construct $T_kx$.  We choose 
$d_{k+1}T_kx$ to be the element $y$ that was just constructed, and because 
of (\yc) we also choose $d_{k+2}T_kx$ to be $y$.  For $0<j<k+1$ we want 
$d_jT_kx$ to be $T_{k-1}d_{j-1}x$ because of (\ya), and for $j>k+2$ we want 
$d_jT_kx$ to be $T_kd_{j-2}x$ because of (\yb).  

We need to check consistency:

\begin{lemma}
\label{l5}
Let 
\[
z_j=
\begin{cases}
T_{k-1}d_{j-1}x & \text{if\, $0<j<k+1$,}\\
y& \text{if\, $j$ is $k+1$ or $k+2$,}\\
T_kd_{j-2}x& \text{if\, $j>k+2$.}
\end{cases}
\]
Then
\begin{equation}
\label{e5}
d_iz_j=d_{j-1}z_i \quad\text{whenever $0<i<j$.}
\end{equation}
\end{lemma}

\begin{proof}[Proof of Lemma \ref{l5}]
First suppose $j<k+1$.  Both sides of \eqref{e5} simplify to
$T_{k-2}d_{j-2}d_{i-1}x$ by (\ya).

If $j=k+1$ then the right side is $d_kT_{k-1}d_{i-1}x$, and the left side
simplifies to this by \eqref{e2}.

If $i<k+1$ and $j=k+2$ both sides simplify to $d_kT_{k-1}d_{i-1}x$, using
\eqref{e2} on the left and (\yc) on the right.

If $i<k+1$ and $j>k+2$ both sides simplify to $T_{k-1}d_{j-3}d_{i-1}x$,
using (\ya) on the left and (\yb) on the right.

If $i=k+1$ and $j=k+2$ both sides are equal to $d_{k+1}y$.

If $i=k+2$ both sides simplify to $d_{k+1}T_kd_{j-2}x$, using (\yc) on the
left and \eqref{e2} on the right.

If $i>k+2$ both sides simplify to $T_kd_{j-3}d_{i-2}x$ using (\yb).
\end{proof}

Now we apply the Kan condition to construct $T_kx$ with $d_jT_kx=z_j$ for
$j\neq 0$. All parts of the inductive hypothesis are true by construction.
\qed 

\section{The \ms\ analogue of Theorem \ref{1}}
\label{s4}

\begin{definition}
Let $\ell\geq 1$.

(i) An $\ell$-fold multi-index $\bf n$ is a sequence of nonnegative integers 
$n_1,\ldots,n_\ell$.  

(ii) For $1\leq p\leq \ell$ let $\bfe_p$ be the $\ell$-fold multi-index with 
1 in the $p$-th position and 0 in all other positions.

(iii) Addition of $\ell$-fold multi-indices is degreewise addtion and 
similarly for subtraction.

(iv) An $\ell$-fold multisemisimplicial set is a collection of sets $X_\bfn$ 
indexed by the $\ell$-fold multi-indices, with maps
\[
d_i^p: X_\bfn\to X_{\bfn-\bfe_p}
\]
for $1\leq p\leq \ell$ and $0\leq i\leq n_p$,
such that 
\begin{alignat*}{2}
d_i^pd_j^p&=d_{j-1}^pd_i^p &\quad&\text{when $i<j$, and}\\
d_i^pd_j^q&=d_j^qd_i^p &\quad&\text{when $p\neq q$.}
\end{alignat*}
\end{definition}

The Kan condition for multisemisimplicial sets is analogous to that for
semisimplicial sets: a simplex can be constructed from a consistent choice of
all but one of its faces.  Here is the formal definition.

\begin{definition}
Let $X$ be an $\ell$-fold \ms\ set.  $X$ satisfies the Kan condition if, for 
every choice of a multi-index $\bf n$, a pair $(r,k)$ with
$1\leq r\leq \ell$ and $0\leq m\leq n_r$, and elements $x^p_i\in 
X_{\bfn-\bfe_p}$ for $(p,i)\neq(r,k)$ satisfying
\[
d_i^px_j^p=d_{j-1}^px_i^p \quad\text{for $i<j$}
\]
and
\[
d_i^px_j^q=d_j^qx_i^p \quad\text{for $p\neq q$},
\]
there exists an element $x\in X_\bfn$ with $d^p_ix=x^p_i$ for $(p,i)\neq
(r,k)$.
\end{definition}

The analogue of Theorem \ref{1} says that a \ms\ set satisfying the Kan
condition has a multisimplicial structure.  Here is the formal statement.

\begin{thm}
\label{2}
Let $X$ be an $\ell$-fold \ms\ set satisfying the Kan condition.  Then there 
are functions
\[
s^q_j:X_\bfn\to X_{\bfn+\bfe_q}
\]
for $0\leq j\leq n_q$, with the following properties.
\begin{alignat}{2}
\label{za}
d_i^qs_j^q&=s_{j-1}^qd_i^q &\quad&\textrm{if\,  $i<j$}.\\
\label{zb}
d_i^qs_j^qx&=x &\quad&\textrm{if\,  $i=j,j+1$}.\\
\label{zc}
d_i^qs_j^q&=s_j^qd_{i-1}^q &\quad&\textrm{if\,  $i>j+1$}.\\
\label{zd}
s_j^qs_i^q&=s_i^qs_{j-1}^q &\quad&\textrm{if\,  $i<j$}. \\
\label{ze}
d_i^ps_j^q&=s_j^qd_i^p &\quad&\text{whenever $p\neq q$}. \\
\label{zf}
s_i^ps_j^q&=s_j^qs_i^p &\quad&\text{whenever $p\neq q$}. 
\end{alignat}
\end{thm}

\begin{remark}
\label{r1}
Notice that to prove (\zf) it suffices (by symmetry) to prove that the 
equation holds for $p>q$.
\end{remark}

The rest of this paper gives the proof of Theorem \ref{2}, which is quite
similar to that of Theorem \ref{1}.
For an element $x\in X_\bfn$, let us define the total degree $|x|$ to be 
$n_1+\ldots+n_\ell$.  We will construct the degeneracies $s^q_j$ by a double 
induction on $|x|$ and the pair $(q,j)$; we order the pairs $(q,j)$ using the 
lexicographic order.

Let $m\geq 0$, $1\leq r\leq \ell$ and $k\geq 0$.
Assume that $s^q_jx$ has been constructed when $|x|<m$, and also
also when $|x|$ is equal to $m$ and
the pair $(q,j)$ is less than $(r,k)$ in the lexicographic order, and that
properties (\za)--(\zf) hold in all relevant cases (that is, in all cases
involving only degeneracies that have already been constructed).  Let $x\in
X_\bfn$ with $n_1+\ldots+n_\ell=m$ and suppose $0\leq k\leq n_r$;
we need to construct $s_k^r x$.

There are three cases, which are dealt with in Sections \ref{s5}--\ref{s7}:

\begin{description}
\item[Case 1] $x$ is in the image of $s_j^q$ for some $q<r$.

\item[Case 2] $x$ is not in the image of $s_j^q$ for $q<r$, but it is in the
image of $s_j^r$ for some $j<k$.

\item[Case 3] $x$ is not in the image of $s_j^q$ when $(q,j)<(r,k)$. 
\end{description}

\section{Proof of Theorem \ref{2}: Case 1}
\label{s5}

Let $(q,j)$ be the smallest pair with $x$ in the image of
$s_j^q$; then
\[
x=s_j^qw
\]
for a unique $w$.  Define
\[
s_k^rx=s_j^qs_k^rw
\]
(as required by (\zf)).   We need to verify (\za)--(\zf).

For property (\za), both sides simplify to $s_j^qs_{k-1}^rd_i^rw$, using (\ze)
and (\za) on the left and (\ze) and (\zf) on the right.  The verifications 
for (\zb) and (\zc) are similar. 

For (\zd) we need:

\begin{lemma}
\label{l6}
If $s_j^qw=s_i^py$ with $p\neq q$ then there is a $v$ with $y=s_j^qv$ and
$w=s_i^pv$.
\end{lemma}

\begin{proof}[Proof of Lemma \ref{l6}]
The proof is similar to that of Lemma \ref{l2}.
Let $v=d_j^qy$.
Then
\[
s_i^py=s_j^qw=s_j^qd_j^qs_j^qw
=s_j^qd_j^qs_i^py=s_j^qs_i^pd_j^qy=s_j^qs_i^pv
=s_i^ps_j^qv,
\]
so $y=s_j^qv$.
Now
\[
s_j^qw=s_i^py=s_i^ps_j^qv=s_j^qs_i^pv
\]
so $w=s_i^pv$.
\end{proof}

Now we verify (\zd).  Let $y\in X_{\bfn-\bfe_r}$ and let $i<k$. Suppose that
$s_i^ry$ is in the image of $s_j^q$ for some $q<r$, and choose the smallest
such pair $(q,j)$;  then
$s_i^ry=s_j^qw$ for some $w$.  Let $v$ be the element given by Lemma \ref{l6}.
Then 
\begin{align*}
s_k^rs_i^ry
&=
s_j^qs_k^rw \quad\text{by definition of $s_k^r$}\\
&=s_j^qs_k^rs_i^rv
=s_j^qs_i^rs_{k-1}^rv\quad\text{by (\zd)}\\
&=s_i^rs_{k-1}^rs_j^qv=s_i^rs_{k-1}^ry
\end{align*}
as required.

For property (\ze), the left side is $d_i^ps_j^qs_k^rw$, and the right side
simplifies to this, using (\ze) and (\zf) when $p\neq q$, (\zc) when $p=q$ and
$i$ is $j$ or $j+1$, and (\ze), (\zf) and (\za) (resp., (\zb)) when $p=q$ and
$i<j$ (resp., $i>j+1$). 

For property (\zf), we want to know (using Remark \ref{r1}) that 
$s_k^rs_i^py=s_i^ps_k^ry$ when $r>p$, 
$y\in X_{\bfn-\bfe_p}$, and 
$s_i^py$ is in the image of $s_j^q$ for some $q<r$. Choose the smallest
such pair $(q,j)$;  then
$s_i^py=s_j^qw$ for some $w$. If $p\neq q$ the result follows easily from Lemma
\ref{l6} and (\zf);
if $p=q$ the proof is similar to that for
(\xd) in Section \ref{s2}.

\section{Proof of Theorem \ref{2}: Case 2}
\label{s6}

Choose the smallest $j$ for which $x$ is in the image of
$s_j^r$; then
\[
x=s_j^rw
\]
for a unique $w$.  Define
\[
s_k^rx=s_j^rs_{k-1}^rw
\]
(as required by (\zd)).   We need to verify (\za)--(\ze); (\zf) is not relevant
for this Case.

The proofs of (\za)--(\zd) are essentially the same as the proofs of 
(\xa)--(\xd) in Section \ref{s2}.

For (\ze), both sides simplify to $s_j^rs_{k-1}^rd_i^pw$, using (\ze) on the
left and (\ze) and (\zd) on the right.

\section{Proof of Theorem \ref{2}: Case 3}
\label{s7}

The argument for this Case is similar to that given in Sections \ref{s3} and
\ref{s3.5}.

\begin{lemma}
\label{l7}
Let $X$ be an $\ell$-fold \ms\ set satisfying the Kan condition.  Then there
are functions
\[
T^q_j:X_\bfn\to X_{\bfn+2\bfe_q}
\]
for $0\leq j\leq n_q$, with the following properties.
\begin{alignat}{2}
\label{wa}
d_i^qT_j^q&=T_{j-1}^qd_{i-1}^q  &\quad&\text{if\, $0<i<j+1$.}\\
\label{wb}
d_i^qT_j^q&=T_j^qd_{i-2}^q &\quad&\text{if\, $i>j+2$.}\\
\label{wc}
d_{j+1}^qT_j^q&=d_{j+2}^qT_j^q &\quad&\text{for all $j$.}\\
\label{wd}
d_0^qd_{j+1}^qT_j^qx&=x &\quad&\text{for all $j$ and $x$.}\\
\label{we}
d_i^pT_j^q&=T_j^qd_i^p &\quad&\text{whenever $p\neq q$.}
\end{alignat}
\end{lemma}

Assuming the lemma, we complete the proof of Theorem \ref{2}.  Given that 
$x$ is not in the image of $s_j^q$ for any pair $(q,j)<(r,k)$, we define 
\[
s_k^rx=d_0^rT_j^rx.
\]
The proofs of (\za)--(\zc) are the same as the proofs of (\xa)--(\xc) at the
end of Section \ref{s3}, and (\zd) is not relevant for this Case.  Property
(\ze) is immediate from \eqref{we}, and (\zf) is not relevant for this Case.
\qed

It remains to prove Lemma \ref{l7}. 
The proof is similar to that of Lemma
\ref{l3} in Section \ref{s3.5}. 

First we define $y_j^q$ for $(q,j)\neq(r,k+1)$ by
\[
y_j^q=
\begin{cases}
x & \text{if\, $q=r$ and $j=0$,}\\
d_k^rT_{k-1}^rd_{j-1}^rx & \text{if\, $q=r$ and $0<j<k+1$,}\\
d_{k+1}^rT_k^rd_{j-1}^rx & \text{if\, $q=r$ and $j>k+1$,}\\
d_{k+1}^rT_k^rd_j^qx & \text{if\, $q\neq r$.}
\end{cases}
\]
The verification that these are consistent is a routine modification of the
argument in Section \ref{s3.5} and is left to the reader.

The Kan condition gives an element $y$ with $d_j^qy=y_j^q$ for
$(q,j)\neq(r,k+1)$.

Now define $z_j^q$ for $(q,j)\neq (r,0)$ by
\[
z_j^q=
\begin{cases}
T_{k-1}^rd_{j-1}^rx & \text{if\, $q=r$ and $0<j<k+1$,}\\
y& \text{if\, $q=r$ and $j$ is $k+1$ or $k+2$,}\\
T_k^rd_{j-2}^rx& \text{if\, $q=r$ and $j>k+2$,}\\
T_k^rd_j^qx& \text{if\, $q\neq r$.}
\end{cases}
\]
Again, the verification that these are consistent is a routine modification of
the corresponding argument in Section \ref{s3.5}.

Now the Kan condition gives an element $T_k^rx$ with $d_j^qT_k^rx=z_j^q$, and
properties \eqref{wa}--\eqref{we} are immediate from the construction.
\qed

\section{A result of Rourke and Sanderson}
\label{s8}

As in \cite{RS}, we will use $|\ |$ for the geometric realization of a
semisimplicial set and $|\ |_M$ for the geometric realization of a
simplicial set (the M stands for Milnor).  

The following result is Corollary
5.4 of \cite{RS}.

\begin{prop}
\label{p1}
Let $Z\to W$
be a semisimplicial inclusion with the property that $|Z|$ is a retract of
$|W|$.  
Let $X$ be a semisimiplicial set satisfying the Kan condition, and let $f:Z\to
X$ be any semisimplicial map.  Then $f$ extends to $W$.
\end{prop}

The rest of this section gives a new proof of this result.

First we need some notation.  The singular complex functor from topological 
spaces to simplicial sets will be denoted by $S$. As in \cite{RS}, 
the free functor from semisimplicial to simplicial sets will be denoted by $G$;
recall that an element of $(GY)_n$ is a pair $(\lambda,y)$, where $y\in Y_p$
for some $p\leq n$ and 
$\lambda:\Delta^n\to\Delta^p$ is a degeneracy or (if $p=n$) the identity map.

There is a natural map of semisimplicial sets 
\[
\alpha:Y\to S|Y|
\]
which takes $y\in Y_n$ to the function $h_y:\Delta^n\to |Y|$ defined by
$h_y(u)=[u,y]$ (where $[u,y]$ denotes the class of $(u,y)$ in $|Y|$).
This extends to a natural map of simplicial sets
\[
\bar{\alpha}:GY\to S|Y|.
\]

\begin{lemma}
\label{l8}
$\bar{\alpha}$ is a trivial cofibration.
\end{lemma}

\begin{proof}[Proof of Lemma \ref{l8}]
Recall that a cofibration of simplicial sets is a 1-1 map.  To see that
$\bar{\alpha}$ is a cofibration, suppose that 
$\bar{\alpha}(\lambda,y)=\bar{\alpha}(\lambda',y')$.  Then
$[\lambda(u),y]=[\lambda'(u),y']$ for every $u$ in the relevant $\Delta^n$.  
Choosing $u$ to be an interior point, we see that $y=y'$.  It follows that 
$\lambda=\lambda'$ on the interior of $\Delta^n$, and hence, by continuity, 
on all of $\Delta^n$.

Next we must show that $\bar{\alpha}$ induces a weak equivalence of
realizations.  For this it suffices to note that the following diagram
commutes:
\[
\xymatrix{
|GY|_M
\ar[r]^{|\bar{\alpha}|}
\ar[d]_\cong
&
|S|Y||_M
\ar[ld]^\simeq
\\
|Y|
&
}
\]
\end{proof}

Now let $Z$, $W$, $X$ and $f$ be as in the statement of Proposition \ref{p1}.
Use Theorem \ref{1} to 
give $X$ a compatible simplicial structure.  
Then the map $f:Z\to X$ extends to a map $\bar{f}:GZ\to X$ of simplicial sets.
Since $\bar{\alpha}$ is a trivial
cofibration of simplicial sets and $X$ is a Kan simplicial set, there is a 
map $\phi:S|Z|\to X$ with $\phi\circ\bar{\alpha}=\bar{f}$.

Next let $r:|W|\to|Z|$ be a retraction and let $g:W\to X$ be the composite
\[
W\to GW\xrightarrow{\bar{\alpha}}
S|W|\xrightarrow{Sr}
S|Z|\xrightarrow{\phi} X.
\]
The commutativity of the following diagram shows that $g$ restricts to $f$ as
required.
\[
\xymatrix{
W
\ar[r]
&
GW
\ar[r]^{\bar{\alpha}}
&
S|W|
\ar[r]^{Sr}
&
S|Z|
\ar[r]^\phi
&
X
\\
Z
\ar[r]
\ar[u]
&
GZ
\ar[u]
\ar[r]^{\bar{\alpha}}
&
S|Z|
\ar[u]
\ar[ru]_=
&&
}
\]

\qed

\smallskip



\providecommand{\bysame}{\leavevmode\hbox to3em{\hrulefill}\thinspace}
\providecommand{\MR}{\relax\ifhmode\unskip\space\fi MR }
\providecommand{\MRhref}[2]{%
  \href{http://www.ams.org/mathscinet-getitem?mr=#1}{#2}
}
\providecommand{\href}[2]{#2}

\end{document}